\documentclass[english,final,
reqno]{amsart}   
\usepackage[T1]{fontenc}
\usepackage[latin1]{inputenc}
\usepackage[english]{babel}
\usepackage{amsmath,amsfonts,amssymb,amsthm}
\usepackage[notcite,notref]{showkeys}
\usepackage{mathrsfs}
\usepackage{graphicx}
\usepackage{fixme}                     

\newtheorem{theorem}{Theorem}[section]
\newtheorem{proposition}[theorem]{Proposition}
\newtheorem{lemma}[theorem]{Lemma}
\newtheorem{corollary}[theorem]{Corollary}
\theoremstyle{definition}
\newtheorem{definition}[theorem]{Definition}
\newtheorem{remark}[theorem]{Remark}

\newtheorem{observation}[theorem]{Observation}

\renewcommand{\phi}{\varphi}
\renewcommand{\epsilon}{\varepsilon}

\newcommand{\Z}{\mathbb Z}

\newcommand{\csa}{$C^*$-al\-ge\-bra}

\newcommand{\shom}{$*$-homo\-mor\-phism}
\newcommand{\siso}{$*$-iso\-mor\-phism}

\renewcommand{\subset}{\subseteq}

\renewcommand{\setminus}{\backslash}
\DeclareMathOperator{\id}{id}

\DeclareMathOperator{\KK}{KK}
\newcommand{\kk}{\mathfrak{KK}}
\DeclareMathOperator{\K}{K}

\DeclareMathOperator{\Sus}{S}
\DeclareMathOperator{\Cone}{C}

\newcommand{\tensor}[2]{{#1}\otimes{#2}}
\DeclareMathOperator{\im}{im}

\DeclareMathOperator{\Hom}{Hom}
\DeclareMathOperator{\Ext}{Ext}

\newcommand{\Op}{\mathcal O}
\newcommand{\LC}{\mathbb{LC}}
\DeclareMathOperator{\FK}{FK}
\newcommand{\NT}{\mathcal{NT}}
\newcommand{\B}{\mathcal B}
\DeclareMathOperator{\Prim}{Prim}
\newcommand{\Ideals}{\mathbb I}
\DeclareMathOperator{\Ch}{Ch}
\DeclareMathOperator{\pd}{pd}

\newcommand{\onto}{\twoheadrightarrow}
\newcommand{\into}{\hookrightarrow}
\newcommand{\op}{\mathrm{op}}
\DeclareMathOperator{\cok}{coker}

\newfont{\cone}{cone}
\renewcommand{\Cone}{\textrm{{\cone 0}\ \ }}

\input xy
\xyoption{all}

\title{Filtrated $\K$-theory for real rank zero $C^*$-algebras}
\author{Sara Arklint}
\address{Department of Mathematical Sciences, University of Copenhagen, Uni\-versi\-tets\-parken~5, DK-2100 Copenhagen, Denmark}
\email{arklint@math.ku.dk}

\author{Gunnar Restorff}
\address{Faculty of Science and Technology, University of Faroe Islands, N\'oat\'un~3, FO-100~T\'orshavn, Faroe Islands}
\email{gunnarr@setur.fo}

\author{Efren Ruiz}
\address{Department of Mathematics, University of Hawaii, Hilo, 200 W.~Kawili St., Hilo, Hawaii, 96720-4091 USA}
\email{ruize@hawaii.edu}

\begin{document}

\bibliographystyle{alpha}

\begin{abstract}
Using Kirchberg $\KK_{X}$-classification of purely infinite, separable, stable, nuclear \csa s with finite primitive ideal space, Bentmann showed that filtrated $\K$-theory classifies purely infinite, separable, stable, nuclear \csa s that satisfy that all simple subquotients are in the bootstrap class and that the primitive ideal space is finite and of a certain type, referred to as accordion spaces.  This result generalizes the results of Meyer-Nest involving finite linearly ordered spaces.  Examples have been provided, for any finite non-accordion space, that isomorphic filtrated $\K$-theory does not imply $\KK_{X}$-equivalence for this class of \csa s.  As a consequence, for any non-accordion space, filtrated $\K$-theory is not a complete invariant for purely infinite, separable, stable, nuclear \csa s that satisfy that all simple subquotients are in the bootstrap class.

In this paper, we investigate the case for real rank zero \csa s and four-point primitive ideal spaces, as this is the smallest size of non-accordion spaces.  Up to homeomorphism, there are ten different connected $T_{0}$-spaces with exactly four points.  We show that filtrated $\K$-theory classifies purely infinite, real rank zero, separable, stable, nuclear \csa s that satisfy that all simple subquotients are in the bootstrap class for eight out of ten of these spaces.
\end{abstract}

\maketitle

\section{Introduction}


The \csa{} classification programme initiated by G.~A.~Elliott in the early seventies has seen a rapid development during the past 20 years. 
%
The notion of real rank zero \csa s introduced by G.~K.~Pedersen and L.~G.~Brown in the late eighties has turn out to be of particular interest in connection with classification of \csa s. Until the mid-nineties most results were concerned with the stably finite algebras, when people such as M.~R\o{}rdam, N.~C.~Phillips, E.~Kirchberg and D.~Huang classified some purely infinite, nuclear, separable \csa s in the bootstrap class. All these had finitely many ideals --- in fact, almost all cases were either the simple case or the one non-trivial ideal case. D.~Huang was also able to classify purely infinite Cuntz-Krieger algebras with finite $\K$-theory (implying that all the $\K_1$-groups are zero). In contrast to the stably finite case, the positive cone of purely infinite \csa s carries no extra information, so it was clear from the beginning, that to classify non-simple purely infinite \csa s one needs to come up with a new invariant, which also encodes the ideal structure and the $\K$-theory of all ideals and quotients. 

The main ingredients of the proof of N.~C.~Phillips and E.~Kirchberg were the UCT of J.~Rosenberg and C.~Schochet and a result saying that every $\KK$-equivalence between (simple, purely infinite, stable, nuclear, separable) \csa s can be lifted to a \siso{} between the algebras. Shortly after, E.~Kirchberg generalized this result to $X$-equivariant $\KK$-theory, where $X$ is (homeomorphic to) the primitive ideal space of the \csa . The only ingredient thus missing to classify purely infinite, nuclear, separable, stable \csa s seemed to be to find the right invariant and prove a UCT for $X$-equivariant $\KK$-theory with this new invariant. For the case with one non-trivial ideal, A.~Bonkat reproved R\o{}rdams result by providing a UCT for this class using the cyclic six-term exact sequence in $K$-theory. The second named author generalized this to two non-trivial ideals by including four cyclic six-term exact sequences. R.~Meyer and R.~Nest, and R.~Bentmann recently proved that the obvious guess of an invariant gives a UCT for certain ideal lattices --- the so-called accordion spaces (including, e.g., all \csa s with exactly three primitive ideals). In turn they also provide a series of counter-examples, where we do not have a UCT. They actually find examples of stable, purely infinite, nuclear, separable \csa s in the bootstrap class with finitely many ideals having isomorphic invariants without being isomorphic. This result seems to be in sharp contrast to the stable classification result for all purely infinite Cuntz-Krieger algebras with finitely many ideals obtained by the second named author by use of methods from shift spaces. 

We find it very likely that the reason that Cuntz-Krieger algebras are classifiable, is the restrictive nature of their $\K$-theory. In this paper we examine what happens to real rank zero algebras in the cases where the primitive ideal space has exactly four points. Moreover, we assume that the space is connected (since otherwise the algebras are direct sums of algebras with fewer than four primitive ideals). Also, all the basic counterexamples of R.~Meyer, R.~Nest, and R.~Bentmann are formulated for algebras with four primitive ideals. 
Up to homeomorphism, there are ten different connected $T_0$-spaces with exactly four points. These are
\begin{align*}
\Op(X_1)&=\{\emptyset,\{4\},\{1,4\},\{2,4\},\{3,4\},\{1,2,4\},\{1,3,4\},\{2,3,4\},\{1,2,3,4\}\}, \\
\Op(X_2)&=\{\emptyset,\{4\},\{3,4\},\{2,3,4\},\{1,3,4\},\{1,2,3,4\}\}, \\
\Op(X_3)&=\{\emptyset,\{4\},\{3,4\},\{2,4\},\{2,3,4\},\{1,2,3,4\}\}, \\
\Op(X_4)&=\{\emptyset,\{1\},\{2\},\{3\},\{1,2\},\{1,3\},\{2,3\},\{1,2,3\},\{1,2,3,4\}\}, \\
\Op(X_5)&=\{\emptyset,\{1\},\{2\},\{1,2\},\{1,2,3\},\{1,2,3,4\}\}, \\
\Op(X_6)&=\{\emptyset,\{3\},\{4\},\{3,4\},\{1,3,4\},\{2,3,4\},\{1,2,3,4\}\}, \\
\Op(X_7)&=\{\emptyset,\{1\},\{1,2\},\{1,2,3\},\{1,2,3,4\}\}, \\
\Op(X_8)&=\{\emptyset,\{1\},\{4\},\{1,2\},\{1,4\},\{1,2,3\},\{1,2,4\},\{1,2,3,4\}\}, \\
\Op(X_9)&=\{\emptyset,\{1\},\{3\},\{1,3\},\{3,4\},\{1,2,3\},\{1,3,4\},\{1,2,3,4\}\}, \\
\Op(X_{10})&=\{\emptyset,\{2\},\{1,2\},\{2,3\},\{1,2,3\},\{2,3,4\},\{1,2,3,4\}\}.
\end{align*}
R.~Meyer and R.~Nest, and R.~Bentmann have proved that the spaces $X_7,X_8,X_9$ and $X_{10}$ have a UCT, and thus we can classify  stable, purely infinite, nuclear, separable \csa s in the bootstrap class with these spaces as primitive ideal spaces. Moreover they have provided counter-examples for classification for all the spaces $X_1,X_2,X_3,X_4,X_5,X_6$. 
In this paper we prove the following

\begin{theorem} \label{thm1}
Let $A$ and $B$ be purely infinite, nuclear, separable \csa s of real rank zero in the bootstrap class of R.~Meyer and R.~Nest (cf.~\cite[4.11]{meyernest_boot}). 
Assume that the primitive ideal space of $A$ and $B$ both are homeomorphic to $X_i$ for an $i=1,2,4,5,7,8,9,10$. 
\begin{itemize}
\item[(1)] If $A$ and $B$ are stable, then every isomorphism from $\FK(A)$ to $\FK(B)$ can be lifted to a \siso{} from $A$ to $B$. 

\item[(2)] If $A$ and $B$ are unital, then every isomorphism from $\FK(A)$ to $\FK(B)$ that preserves the unit can be lifted to a \siso{} from $A$ to $B$.
\end{itemize} 
\end{theorem}

\begin{theorem} \label{thm2}
There exist stable, purely infinite, nuclear, separable \csa s of real rank zero in the bootstrap class of R.~Meyer and R.~Nest (cf.~\cite[4.11]{meyernest_boot}) with the primitive ideal space homeomorphic to $X_3$, which have isomorphic filtrated $\K$-theory without being isomorphic. 
\end{theorem}
\noindent where $\FK$ denotes the functor filtrated $\K$-theory which will be defined shortly.

For the case where the primitive ideal space is isomorphic to $X_6$ there are still no counterexamples for the real rank zero case --- however our methods do not apply as there is no known finite refinement of $\FK$ which gives a UCT. 

In general the unital part of Theorem~\ref{thm1} follows from the stable part by using results from \cite{restorffruiz}. For $X_7$, Theorem~\ref{thm1} is proved by R.~Meyer and R.~Nest in~\cite[4.14]{filtrated}, for $X_8$, $X_9$ and $X_{10}$, it is proved by R.~Bentmann in~\cite[5.4.2]{rasmus}. In Section~2 of this paper we set up notation and prove some preliminary results used later in this paper. In Sections~3 and~4 Theorem~\ref{thm1} is proved for $X_1$, $X_2$, $X_4$ and $X_5$ (cf.\ Corollaries~\ref{classi1} and~\ref{classi2} and Remarks~\ref{Xoneop} and~\ref{Xtwoop}). The proofs rely on the result~\cite[4.3]{kirchberg} of E.~Kirchberg that $\KK(X)$-equi\-va\-len\-ces lift to $X$-equivariant isomorphisms for stable, separable, nuclear, purely infinite \csa s with primitive ideal space homeomorphic to a finite $T_0$-space $X$. Theorem~\ref{thm2} is proved in Section~\ref{secthm2}.

\section{Preliminaries and notation}

In this section, we briefly discuss \csa s over a topological space $X$ and the invariant introduced by R.~Meyer and R.~Nest in \cite{filtrated} called filtrated $\K$-theory.  We refer the reader to \cite{filtrated} for details.

We would like to note that there are other invariants in the literature which are closely related to filtrated $\K$-theory.  Examples are filtered $\K$-theory and ideal related $\K$-theory.  It has been proved by R.~Meyer and R.~Nest in \cite{filtrated} and R.~Bentmann in~\cite[5.4.2]{rasmus} that for the spaces $X_{i}$ that these invariants are naturally isomorphic to filtrated $\K$-theory.  It is not known if these invariants are naturally isomorphic for all finite topological spaces. 


\subsection{\csa s over a topological space $X$}

A \csa\ over a topological space $X$ is a pair $(A,\psi)$ consisting of a \csa\ $A$ and a continuous map $\psi\colon\Prim(A)\to X$ where $\Prim(A)$ denotes the primitive ideal space of $A$.
Assume from now on that $X$ is a finite topological space satisfying the $T_0$ separation axiom, i.e., such that $\overline{\{x\}}\neq\overline{\{y\}}$ for all $x,y\in X$ with $x\neq y$.
Let $\Op(X)$ denote the open subsets of $X$, and let $\Ideals(A)$ denote the lattice of (two-sided, closed) ideals of $A$.
A \csa\ over $X$ can then equivalently be defined as a pair $(A,\psi)$ consisting of a \csa\ $A$ and a map $\psi\colon\Op(X)\to\Ideals(A)$ that preserves infima and suprema.  We then write $A(U)$ for $\psi(U)$.

The locally closed subsets of $X$ are denoted by $\LC(X)=\{ U\setminus V \mid  V,U\in\Op(X), V\subset U\}$, and the connected, non-empty, locally closed subsets of $X$ are denoted by $\LC(X)^*$.
For $Y\in\LC(X)$ we define $A(Y)=A(U)/A(V)$ when $Y=U\setminus V$ for some $V,U\in\Op(X)$ satisfying $V\subset U$.  Up to canonical isomorphism, $A(Y)$ does not depend on the choice of $U$ and $V$.

For \csa s $A$ and $B$ over $X$, we say that a \shom\ $\phi\colon A\to B$ is $X$-equivariant if $\phi(A(U))\subset B(U)$ holds for all $U\in\Op(X)$.
An extension $A\into B\onto C$ is called $X$-equivariant if it induces an extension $A(U)\into B(U)\onto C(U)$ for all $U\in\Op(X)$.

E.~Kirchberg has constructed $X$-equivariant $\KK$-theory $\KK_*(X;-,-)$, also called ideal related $\KK$-theory and here referred to as $\KK(X)$-theory.  We denote by $\kk(X)$ the category of separable \csa s over $X$ with $\KK_0(X)$-classes as morphism groups. In~\cite[3.11]{meyernest_boot}, R.~Meyer and R.~Nest show that the category $\kk(X)$ equipped with the suspension automorphism $\Sus$ and mapping cone triangles as distinguished triangles is triangulated; so mapping cones of $X$-equivariant \shom s give exact triangles, and so do extensions over $X$ that split by an $X$-equivariant completely positive contraction.  

\subsection{Filtrated $\K$-theory $\FK$ and the UCT}
One defines for each $Y\in\Op (X)$ the functor $\FK_Y$ by $\FK_Y(A)=\K_*(A(Y))$.
We write $\FK_Y^i(A)$ for $\K_i(A(Y))$.
In~\cite{filtrated} R.~Meyer and R.~Nest construct commutative \csa s $R_Y$ over $X$ such that $\KK_*(X;R_Y,-)$ and $\FK_Y$ are equivalent functors.

By the Yoneda Lemma, cf.~\cite[3.2]{working}, the set $\NT(Y,Z)$ of natural trans\-forma\-tions from the functor $\FK_Y$ to the functor $\FK_Z$ is then given by $\KK_*(X;R_Z,R_Y)$.
Given $\alpha\in\KK_*(X;R_Z,R_Y)$ we denote by $\bar\alpha$ the corresponding element in $\NT(Y,Z)$ given by $\alpha\boxtimes -$ where $-\boxtimes-$ denotes the $X$-equivariant Kasparov product.  Given $f\in\NT(Y,Z)$, we let $\hat f$ denote the corresponding element in $\KK_*(X;R_Z,R_Y)$.

The functor $\FK$ is then defined as the family of functors $(\FK_Y)_{Y\in\LC(X)^*}$ together with the sets $\NT(Y,Z)$ of natural transformations.
The target category of $\FK$ is the category of modules over the ring $\NT=\bigoplus_{Y,Z\in\LC(X)^*} \NT(Y,Z)$.
A homomorphism $\FK(A)\to\FK(B)$ is then a family of homomorphisms $(\phi_Y)$ that respects the natural transformations.
Kasparov multiplication induces a map $\KK_*(X;A,B)\to\Hom(\FK(A),\FK(B))$, and for $A=R_Y$ this map is an isomorphism.

In~\cite{filtrated} R.~Meyer and R.~Nest establish a UCT for $\KK(X)$-theory, i.e., they establish exactness of
\[ \Ext_{\NT}^1(\FK(A),\FK(B)) \into \KK_*(X;A,B) \onto \Hom_{\NT}(\FK(A),\FK(B)) \]
for $A$ and $B$ separable \csa s over $X$ with $A$ belonging to the bootstrap class $\B(X)$ defined by R.~Meyer and R.~Nest, cf.~\cite[4.11]{meyernest_boot}, and with $\FK(A)$ having projective dimension at most 1 as a module over $\NT$.
By construction, $\FK(R_Y)$ has projective dimension 0 for all $Y\in\LC(X)^*$.
By~\cite[4.13]{meyernest_boot}, a nuclear \csa\ over $X$ belongs to $\B(X)$ if and only if its simple subquotients belong to the bootstrap class of J.~Rosenberg and C.~Schochet.

\subsection{Construction of $R_Y$}

The \csa s $R_Y$ are constructed as follows.  Define a partial order on $X$ by $x\leq y$ when $x\in\overline{\{y\}}$.  The order complex $\Ch(X)$ is the geometric realisation of the simplicial set whose nondegenerate $n$-simplices $[x_0,\ldots,x_n]$ are strict chains $x_0<\cdots<x_n$.
Maps $m,M\colon \Ch(X)\to X$ is then defined by the inner of a simplex $[x_0,\ldots,x_n]$ being sent to $x_0$ by $m$ and to $X_n$ by $M$.  The \csa s $R_Y$ over $X$ are then defined by
$R_Y(Z)=C_0(m^{-1}(Y)\cap M^{-1}(Z))$ for all $Y,Z\in\LC(X)$.

For $Y\in\LC(X)$ and $U\in\Op(Y)$, we then get $X$-equivariant extensions $R_{Y\setminus U}\into R_Y\onto R_U$.  The natural transformation given by $R_{Y\setminus U}\into R_Y$ is denoted $r_Y^{Y\setminus U}$ and called a restriction map, the natural transformation given by $R_Y\onto R_U$ is denoted by $i_U^Y$ and called an extension map, and the natural transformation given by $R_{Y\setminus U}\into R_Y\onto R_U$ is denoted by $\delta_{Y\setminus U}^U$ and called a boundary map.
For a \csa\ $A$ over $X$, these natural transformations are the ones appearing in the six-term exact sequence induced by the extension $A(U)\into A(Y)\onto A(Y\setminus U)$.
It is unknown whether there exists finite $T_0$-spaces $X$ over which the ring $\NT$ is not generated by transformations of this form, but for the spaces $X_1,X_2,\ldots,X_{10}$ considered in this paper, this is not the case.

\section{The counterexample of Meyer and Nest}
We now restrict to the space $X_1=\{1,2,3,4\}$ with $\Op(X_1)=\{\emptyset\}\cup\{U\subset X_1\mid 4\in U\}$.  We abbreviate, e.g., $\{1,2,3\}$ to $123$.
A \csa\ $A$ over $X_1$ is then an extension of the form $A(4)\into A\onto A(1)\oplus A(2) \oplus A(3)$.
The ordering on $X$ induced by its topology is then defined by $i\leq 4$ for all $i\in X_1$, its Hasse diagram (or, more correctly, the Hasse diagram of the inverse order relation) is
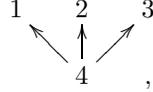
\[ \xymatrix@!0@C=25pt@R=25pt{
1 & 2 & 3 \\
& 4\ar[ul]\ar[u]\ar[ur] & ,
} \]
and $\LC(X_1)^*=\{4,14,24,34,124,134,234,1234,1,2,3\}$.
In~\cite{filtrated} it is shown that the ring $\NT=\bigoplus_{Y,Z\in\LC(X_1)^*}\NT(Y,Z)$ is generated by natural transformations $i$, $r$ and $\delta$ that are induced by six-term exact sequences, and the indecomposable transformations are of infinite order and fit into the following diagram
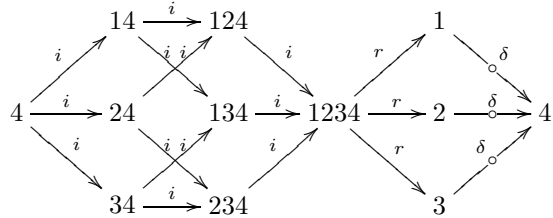
\[ \xymatrix@!0@C=40pt@R=35pt{
& 14\ar[r]^-{i}\ar[rd]^-{i} & 124\ar[rd]^-{i} & & 1\ar[dr]^-{\delta}|\circ & \\
4\ar[ru]^-{i}\ar[r]^-{i}\ar[rd]^-{i} & 24\ar[ru]^-{i}\ar[rd]^-{i} & 134\ar[r]^-{i} & 1234\ar[ru]^-{r}\ar[r]^-{r}\ar[rd]^-{r} & 2\ar[r]^-{\delta}|\circ & 4 \\
& 34\ar[ru]^-{i}\ar[r]^-{i} & 234\ar[ru]^-{i} & & 3\ar[ur]^-{\delta}|\circ &
} \]
where the six squares commute and the sum of the three transformations from 1234 to 4 vanishes.

\subsection{The refined invariant}
In~\cite{filtrated}, R.~Meyer and R.~Nest refine the invariant $\FK$ to an invariant $\FK'$.  They prove a UCT for this refined invariant, so for $A$ and $B$ in the bootstrap class $\B(X_1)$ one can lift isomorphisms between $\FK'(A)$ and $\FK'(B)$ to $\KK(X_1)$-equivalences, and by combining this with the classification result~\cite[4.3]{kirchberg} of E.~Kirchberg conclude that it strongly classifies the stable, purely infinite, separable, nuclear \csa s $A$ that are tight over $X_1$ and whose simple subquotients $A(4)$, $A(1)$, $A(2)$ and $A(3)$ lie in the bootstrap class, see~\cite[5.14, 5.15]{filtrated}. 

In \cite{restorffruiz}, the second and third author showed how one can strongly classify a class of unital properly infinite $C^{*}$-algebras given that the this class are strongly classified up to stable isomorphism.  Since $\FK' ( \cdot )$ strongly classifies the class of stable, purely infinite, separable, nuclear \csa s $A$ that are tight over $X_{1}$, by Theorem 2.1 of \cite{restorffruiz}, $\FK'( \cdot )$ together with class of the unit strongly classifies the class of unital, purely infinite, separable, nuclear \csa s $A$ that are tight over $X_{1}$.

The invariant is defined by constructing a \csa\ $R_{12344}$ over $X_1$ and adding $\KK_*(X_1;R_{12344},-)$ to the family of functors.
The indecomposable transformations in the larger ring $\NT'=\bigoplus_{Y,Z\in\LC(X_1)^*\cup\{12344\}}\NT(Y,Z)$ fit into the following diagram:
\begin{equation}
 \vcenter{\xymatrix@!0@C=40pt@R=35pt{
& 14\ar[rd] & & 124\ar[rd]^-{i} & & 1\ar[dr]^-{\delta}|\circ & \\
4\ar[ru]^-{i}\ar[r]^-{i}\ar[rd]^-{i} & 24\ar[r] & 12344\ar[ru]\ar[r]\ar[rd] & 134\ar[r]^-{i} & 1234\ar[ru]^-{r}\ar[r]^-{r}\ar[rd]^-{r} & 2\ar[r]^-{\delta}|\circ & 4 \\
& 34\ar[ru] & & 234\ar[ru]^-{i} & & 3\ar[ur]^-{\delta}|\circ &
}} \label{diagram1}
\end{equation}

The \csa\ $R_{12344}$ is the mapping cone of a generator of the cyclic free group $\NT(234,14)$ and its filtrated $\K$-theory is
\begin{equation} 
\vcenter{\xymatrix@!0@C=40pt@R=35pt{
& 0\ar[r]^-{i}\ar[rd]^-{i} & \Z\ar[rd]^-{i} & & \Z\ar[dr]^-{\delta}|\circ & \\
\Z[1]\ar[ru]^-{i}\ar[r]^-{i}\ar[rd]^-{i} & 0\ar[ru]^-{i}\ar[rd]^-{i} & \Z\ar[r]^-{i} & \Z^2\ar[ru]^-{r}\ar[r]^-{r}\ar[rd]^-{r} & \Z\ar[r]^-{\delta}|\circ & \Z[1] \\
& 0\ar[ru]^-{i}\ar[r]^-{i} & \Z\ar[ru]^-{i} & & \Z\ar[ur]^-{\delta}|\circ &
 }} \label{diagramgroups}
 \end{equation}
where the three maps $i_{ij4}^{1234}$ are given by the three coordinate embeddings $\Z\to\Z^3/(1,1,1)$, the three maps $r_{1234}^k$ are given by the three projections $\Z^3/(1,1,1)\to\Z^2/(1,1)$ onto coordinate hyperplanes, and the three maps $\delta_k^4$ are the identity.


\begin{lemma} \label{thankyourasmus}
	Assume that $\FK_Y(A)$ and $\FK_Y(R_{12344})$ are isomorphic for all $Y\in\LC(X_1)^*$ and that $i_{124}^{1234}\oplus i_{134}^{1234}\colon \FK_{124}(A)\oplus\FK_{134}(A)\to\FK_{1234}(A)$ is an isomorphism.  
	Then $\FK(A)$ and $\FK(R_{12344})$ are isomorphic as $\NT$-modules and $A$ and $R_{12344}$ are $\KK(X_1)$-equivalent.
\begin{proof}
Define for each $Y\in\LC(X)^*$ an $\NT$-module $P_Y$ as $P_Y(Z)=\NT(Y,Z)$.  Then $P_Y$ is freely generated by $\id_Y\in P_Y(Y)$ as an $\NT$-module.
Define $j\colon P_{1234}\to P_{124}\oplus P_{134}\oplus P_{234}$ by $f\mapsto fi_{124}^{1234} +fi_{134}^{1234} +fi_{234}^{1234}$.  Then $\FK(R_{12344})$ is isomorphic to $\cok j$ as $\NT$-modules, cf.~\cite[Section~5]{filtrated}, with $\im j$ generated by $i_{124}^{1234} +i_{134}^{1234} +i_{234}^{1234}$.

Hence an $\NT$-morphism $\FK(R_{12344})\to\FK(A)$ can be defined by choosing elements $g_Y\in\FK_Y(A)$, $Y\in\{124,134,234\}$, satisfying $i_{124}^{1234}(g_{124}) +i_{134}^{1234}(g_{134}) +i_{234}^{1234}(g_{234})=0$, and defining the map by $\id_Y\mapsto g_Y$ for $Y\in\{124,134,234\}$ and expanding by $\NT$-linearity.

If $g_Y$ generates $\FK_Y(A)$ for all $Y\in\{124,134,234\}$, then the morphism will be an isomorphism: it is automatically bijective $\FK_Z(R_{12344})\to\FK_Z(A)$ for $Z\in\{124,134,234\}$, by the assumptions in the lemma it is therefore surjective and hence bijective for $Z=1234$, and by exactness it then follows that it is bijective for $Z\in\{1,2,3\}$ whereby bijectivity for $Z=4$ also follows, cf.~the Diagram~\eqref{diagramgroups}.

Let $g_Y$ be a generator of $\FK_Y(A)$ for $Y\in\{124,134,234\}$.  Since $i_{124}^{1234} \oplus i_{134}^{1234}$ is an isomorphism, $\FK_{1234}(A)$ is spanned by $i_{124}^{1234}(g_{124})$ and $i_{134}^{1234}(g_{134})$ so we may write $i_{234}^{1234}(g_{234})=mi_{124}^{1234}(g_{124})+ni_{134}^{1234}(g_{134})$ for some $m,n\in\Z$.  Since $\FK_{34}(A)=0$, $\FK_{24}(A)=0$ and $\FK_{14}(A)=0$, the four maps $r_{234}^2\colon\FK_{234}(A)\to\FK_2(A)$, $r_{234}^3\colon\FK_{234}(A)\to\FK_3(A)$, $r_{124}^2\colon\FK_{124}(A)\to\FK_2(A)$  and $r_{134}^3\colon\FK_{134}(A)\to\FK_3(A)$ are isomorphisms, so $r_{124}^2(g_{124})$ and $r_{234}^2(g_{234})=mr_{124}^2(g_{124})$ both generate $\FK_2(A)$, and $r_{134}^3(g_{134})$ and $r_{234}^3(g_{234})=nr_{134}^3(g_{134})$ both generate $\FK_3(A)$, so $m,n\in\{\pm 1\}$.  By replacing $g_{124}$ with $-mg_{124}$ and $g_{134}$ with $-ng_{134}$ the required is fulfilled.

In the discussion after the proof of Lemma 5.9 in \cite{filtrated}, we have that the natural homomorphism from $\KK ( X_{1} ; R_{12344} , A )$ to $\mathrm{Hom} ( \FK ( R_{12344} ), \FK ( A ) )$ is an isomorphism.  Since $\FK(A)$ and $\FK(R_{12344})$ are isomorphic as $\NT$-modules, we have that $A$ and $R_{12344}$ are $\KK(X_1)$-equivalent.
\end{proof}
\end{lemma}

\begin{lemma} \label{kegletrick1}
	There exists an exact triangle
\[ \xymatrix@!0@C=40pt@R=35pt{
R_{1234}\ar[dr]_-{\phi} & & R_{12344}\ar[ll]_-{\pi}|\circ \\
& R_{124}\oplus R_{134}\oplus R_{234}\ar[ur]_-{\iota} & 
} \]
	satisfying that $\bar\phi = (i_{124}^{1234},i_{134}^{1234},i_{234}^{1234})\in\bigoplus\NT(ij4,1234)$,
	that $\bar\pi$ generates the group $\NT(1234,12344)$, and that $\bar\iota=(f^{124},f^{134},f^{234})\in\bigoplus\NT(12344,ij4)$ with each $f^{ij4}$ generating $\NT(12344,ij4)$ respectively.
	\begin{proof}
	Let $\phi\colon R_{1234}\to R_{124}\oplus R_{134}\oplus R_{234}$ be given by restriction to subsets; then $\bar\phi=(i_{124}^{1234},i_{134}^{1234},i_{234}^{1234})$.  Constructing the mapping cone $\Cone_\phi$ of $\phi$, we get an exact triangle
	\[ \xymatrix@!0@C=40pt@R=35pt{
	R_{1234} \ar[dr]_-{\phi} & & \Sus\Cone_\phi\ar[ll]_-{\pi}|\circ \\
	& R_{124}\oplus R_{134}\oplus R_{234}\ar[ur]_-{\iota} & 
	} \]
	and by applying $\KK_*(X_1;R_Y,-)=\FK_Y$ and calculating $\FK_Y(\phi)$, one sees that $\FK_Y(\Sus\Cone_\phi)$ and $\FK_Y(R_{12344})$ are isomorphic for all $Y\in\LC(X_1)^*$, cf.\ Diagram~\eqref{diagramgroups}.
	
Furthermore one sees that $\FK_{ij4}(\iota)$ are isomorphisms, and that $\FK_{1234}(\iota)$ is surjective as $\FK_{1234}(\phi)$ is injective and (using standard generators) is given by $\Z\ni x\mapsto (x,x,x)\in\Z\oplus\Z\oplus\Z$.
Using that $\FK_Y(\iota)$ respects the natural transformations, and that the natural transformation $\FK_{124}(\bigoplus R_{ij4})\oplus\FK_{134}(\bigoplus R_{ij4})\to\FK_{1234}(\bigoplus R_{ij4})$ is given by $\Z\oplus\Z\ni (x,y)\mapsto (x,y,0)\in\Z\oplus\Z\oplus\Z$ (using standard generators), one can then check that $\FK_{124}(\Sus\Cone_\phi)\oplus\FK_{134}(\Sus\Cone_\phi)\to\FK_{1234}(\Sus\Cone_\phi)$ is an isomorphism. Hence $\Sus\Cone_\phi$ and $R_{12344}$ are $\KK(X_1)$-equivalent by Lemma~\ref{thankyourasmus}.
	
	Therefore $\pi$ and $\iota$ induce natural transformations, and since all the involved groups of natural transformations are cyclic and free, we may write $\bar\pi=nf_{1234}$ with $f_{1234}$ generating $\NT(1234,12344)$ and $\bar\iota=(n_{ij4}f^{ij4})$ with $f^{ij4}$ generating the group $\NT(12344,ij4)$.
	
	Then $\FK_{ij4}(\iota)=n_{ij4}\FK_{ij4}(\hat f^{ij4})$, since $\FK_{ij4}(R_{124}\oplus R_{134}\oplus R_{234})=\FK_{ij4}(R_{ij4})$, so as $\FK_{ij4}(\iota)$ is an isomorphism $\Z\to\Z$, we see that $n_{ij4}=\pm 1$.
	
	But $\FK_Y(\pi)=0$ for all $Y$.  However, since $\FK_{ij4}(R_{1234})=0$ and $\FK_{1234}(R_{1234})=\Z$, we get by applying $\KK_*(X_1;-,R_{1234})$ to the exact triangle that $\bar\pi=nf_{1234}$ on $R_{1234}$ is an isomorphism $\Z\to\Z[1]$, hence $n=\pm 1$.
	\end{proof}
\end{lemma}

\begin{lemma} \label{kegletrick2}
	There exists an exact triangle
\[ \xymatrix@!0@C=40pt@R=35pt{
R_{12344}\ar[dr]_-{\pi} & & R_{4}\ar[ll]_-{\iota}|\circ \\
& R_{14}\oplus R_{24}\oplus R_{34}\ar[ur]_-{\phi} & 
} 
\]
satisfying that $\bar\phi = (i_{4}^{14},i_{4}^{24},i_{4}^{34})\in\bigoplus\NT(4,k4)$,
	that $\bar\iota$ generates $\NT(12344,4)$, and that $\bar\pi=(f_{14},f_{24},f_{34})\in\bigoplus\NT(k4,12344)$ with each $f_{k4}$ generating the group $\NT(k4,12344)$ respectively.
\begin{proof}
Let $\phi\colon R_{14}\oplus R_{24}\oplus R_{34}\to M_3(R_4)$ be given by restriction to subsets such that $\bar\phi = (i_{4}^{14},i_{4}^{24},i_{4}^{34})$ and construct the mapping cone $\Cone_\phi$ of $\phi$.  By calculating $\FK_Y(\phi)$ and by applying $\FK_Y$ to the mapping cone triangle
\[ \xymatrix@!0@C=40pt@R=35pt{
\Cone_\phi\ar[dr]_-{\pi} & & R_{4}\ar[ll]_-{\iota}|\circ \\
& R_{14}\oplus R_{24}\oplus R_{34}\ar[ur]_-{\phi} & 
} 
\]
we see that $\FK_Y(\Cone_\phi)\cong\FK_Y(R_{12344})$ for all $Y\in\LC(X_1)^*$.

Furthermore we see that $\FK_4(\iota)$ and $\FK_k(\pi)$ are isomorphisms, and that $\FK_{ij4}(\pi)$ and $\FK_{1234}(\pi)$ are injective as $\FK_{ij4}(\phi)$ and $\FK_{1234}(\phi)$ are surjective and (by using standard generators) are given by $\Z\oplus\Z\ni (x,y)\mapsto x+y\in\Z$ respectively $\Z\oplus\Z\oplus\Z\ni (x,y,z)\mapsto x+y+z\in\Z$.

Using that $\FK_Y(\pi)$ respects the natural transformations, and that the natural transformation $\FK_{124}(\bigoplus R_{k4})\oplus\FK_{134}(\bigoplus R_{k4})\to\FK_{1234}(\bigoplus R_{k4})$ is given by $(\Z\oplus\Z)\oplus(\Z\oplus\Z)\ni (x,y,z,w)\mapsto (x+z,y,w)\in\Z\oplus\Z\oplus\Z$ (using standard generators), one can then check that $\FK_{124}(\Cone_\phi)\oplus\FK_{134}(\Cone_\phi)\to\FK_{1234}(\Cone_\phi)$ is an isomorphism.
	Hence $\Cone_\phi$ and $R_{12344}$ are $\KK(X_1)$-equivalent by Lemma~\ref{thankyourasmus}.
	
	Therefore $\pi$ and $\iota$ induce natural transformations, so we may write $\bar\iota=nf^4$ with $f^4$ generating $\NT(12344,4)$ and $\bar\pi=(n_{k4}f_{k4})$ with $f_{k4}$ generating the group $\NT(k4,12344)$.

	As $\FK_4(\iota)=n\FK_4(\hat f^4)$ is an isomorphism $\Z\to\Z[1]$, we see that $n=\pm 1$. 
	And as $\FK_k(R_{14}\oplus R_{24}\oplus R_{34})=\FK_k(R_k)$, we see that $\FK_k(\pi)=n_{k4}\FK_k(\hat f_{k4})$, so as $\FK_k(\pi)$ is an isomorphism $\Z\to\Z$, we see that $n_{k4}=\pm 1$.
\end{proof}
\end{lemma}

\begin{lemma} \label{gunnarsseqs}
	There exist natural transformations $f_{14},f_{24},f_{34},f^{124},f^{134},f^{234}$ 
	such that $\left<f_{k4}\right>=\NT(k4,12344)$ and $\left<f^{ij4}\right>=\NT(12344,ij4)$ and such that the sequences
\[ \xymatrix@!0@C=60pt@R=35pt{
\FK_{1234}(A)\ar[rr]^-{f_{m4} i_4^{m4}\delta_n^4r_{1234}^n}|\circ & & \FK_{12344}(A)\ar[dl]^-{(f^{ij4})} \\
& \FK_{124}(A)\oplus\FK_{134}(A)\oplus\FK_{234}(A)\ar[ul]^-{(i_{ij4}^{1234})} & 
} 
\] and \[
\xymatrix@!0@C=60pt@R=35pt{
\FK_{12344}(A)\ar[rr]^-{\delta_m^4r_{mn4}^mf^{mn4}}|\circ & & \FK_4(A)\ar[dl]^-{(i_4^{k4})} \\
& \FK_{14}(A)\oplus\FK_{24}(A)\oplus\FK_{34}(A)\ar[ul]^-{(f_{k4})} & 
} 
\]	
	are exact for all \csa s $A$ over $X_1$ and all $m,n\in\{1,2,3\}$ with $m\neq n$.
\begin{proof}
	This follows from Lemmas~\ref{kegletrick1} and~\ref{kegletrick2} by applying $\KK_*(X_1;-,A)$ and using that by the Diagram~\eqref{diagram1} the transformation $f_{m4} i_4^{m4}\delta_n^4r_{1234}^n$ generates the group $\NT(1234,12344)$ and $\delta_m^4r_{mn4}^mf^{mn4}$ generates $\NT(12344,4)$.
\end{proof}
\end{lemma}

\subsection{A classification result} \label{class_first}

\begin{proposition} \label{gunnarefren}
	Let $A$ and $B$ be \csa s over $X_1$ and 
	assume that the maps $\delta_m^4\colon\FK_m^{n}(A)\to\FK_4^{1-n}(A)$ and  $\delta_m^4\colon\FK_m^{n}(B)\to\FK_4^{1-n}(B)$ vanish for some $m\in\{1,2,3\}$ and some
	 $n\in\{0,1\}$.
	Then any homomorphism $\phi\colon\FK(A)\to\FK(B)$ can be uniquely extended to a homomorphism $\phi'\colon\FK'(A)\to\FK'(B)$.
	Furthermore, if $\phi$ is an isomorphism then so is $\phi'$.
\begin{proof}
	Let $\phi\colon\FK(A)\to\FK(B)$ be a homomorphism.	  
	We may extend it by defining $\phi_{12344}\colon\FK_{12344}(A)\to\FK_{12344}(B)$ by the following diagrams:
\[ \xymatrix{
	0\ar[r] & \FK_{12344}^{1-n}(A)\ar[r]^-{(f^{ij4})}\ar@{..>}[d]^-{\phi_{12344}^{1-n}} & \FK_{124}^{1-n}(A)\oplus\FK_{134}^{1-n}(A)\oplus\FK_{234}^{1-n}(A)\ar[r]^-{(i_{ij4}^{1234})}\ar[d]^-{\phi_{124}^{1-n}\oplus\phi_{134}^{1-n}\oplus\phi_{234}^{1-n}} & \FK_{1234}^{1-n}(A)\ar[d]^-{\phi_{1234}^{1-n}} \\
	0\ar[r] & \FK_{12344}^{1-n}(B)\ar[r]^-{(f^{ij4})} & \FK_{124}^{1-n}(B)\oplus\FK_{134}^{1-n}(B)\oplus\FK_{234}^{1-n}(B)\ar[r]^-{(i_{ij4}^{1234})} & \FK_{1234}^{1-n}(B)
} \]
\[ \xymatrix{
	\FK_4^n(A)\ar[r]^-{(i_4^{k4})}\ar[d]^-{\phi_4^n} & \FK_{14}^n(A)\oplus\FK_{24}^n(A)\oplus\FK_{34}^n(A)\ar[r]^-{(f_{k4})}\ar[d]^-{\phi_{14}^n\oplus\phi_{24}^n\oplus\phi_{34}^n} & \FK_{12344}^n(A)\ar[r]\ar@{..>}[d]^-{\phi_{12344}^n} & 0 \\ 
		\FK_4^n(B)\ar[r]^-{(i_4^{k4})} & \FK_{14}^n(B)\oplus\FK_{24}^n(B)\oplus\FK_{34}^n(B)\ar[r]^-{(f_{k4})} & \FK_{12344}^n(B)\ar[r] & 0 
} \]
	By Lemma~\ref{gunnarsseqs} the four horizontal sequences in the diagrams are exact, hence $\phi_{12344}$ is well-defined and is bijective if $\phi$ is an isomorphism. 
	
	By construction $\phi_{12344}^n$ respects the natural transformations $f_{k4}$ and $\phi_{12344}^{1-n}$ respects the naturals transformations $f^{ij4}$.
	Since $(f^{ij4})$ is injective on $\FK^{1-n}_{12344}(B)$ and since $(f_{k4})$ is surjective on $\FK^n_{12344}(A)$, 
	it suffices to check that
	\[ (f^{ij4})\phi^{1-n}_{12344}f_{k4} = (f^{ij4})f_{k4}\phi^{1-n}_{k4} \quad\text{and}\quad  \phi^n_{ij4}f^{ij4}(f_{k4}) = f^{ij4}\phi^n_{12344}(f_{k4}). \]  
And 	this holds by construction of $\phi_{12344}$ as
	\[ f^{ij4}f_{k4}\phi_{k4} = \phi_{ij4}f^{ij4}f_{k4} \]
	since $f^{ij4}f_{k4}\in\NT(k4,ij4)$.
	
	Since the natural transformations in $\FK'$ are generated by the natural transformations in $\FK$ together with the natural transformations $f_{k4}$ and $f^{ij4}$, we see that the extended $\phi$ respects all the natural transformations in $\FK'$, hence it is an $\NT'$-morphism between $\FK'(A)$ and $\FK'(B)$.
\end{proof}
\end{proposition}
\begin{observation} \label{realrankzero}
A separable, nuclear, purely infinite, tight \csa\ $A$ over a finite $T_0$-space $X$ is of real rank zero if and only if the boundary map $\delta_{Y\setminus U}^U$ vanishes on $\K_0(A(Y\setminus U))$ for all $Y\in\LC(X)$ and all $U\in\Op(Y)$.
This follows from the fact that all Kirchberg algebras have real rank zero combined with the following result of L.~G.~Brown and G.~K.~Pedersen, cf.~\cite[3.14]{brownpedersen}:  Given an extension $I\into B\onto B/I$ of \csa s, $B$ has real rank zero if and only if $I$ and $B/I$ have real rank zero and projections in $B/I$ lift to projections in $B$.
\end{observation}
\begin{corollary} \label{isoslift1}
	Let $A$ and $B$ be \csa s in the bootstrap class over $X_1$ and with $A$ of real rank zero.
	Then any isomorphism between $\FK(A)$ and $\FK(B)$ lifts to a $\KK(X_1)$-equivalence.
\begin{proof}
	Since $A$ is of real rank zero, $\delta_2^4\colon\FK_2^0(A)\to\FK_4^1(A)$ vanishes by~\cite[3.14]{brownpedersen}, and since $\FK(A)$ and $\FK(B)$ are isomorphic,
	 $\delta_2^4\colon\FK_2^0(B)\to\FK_4^1(B)$ also vanishes. 
	By Proposition~\ref{gunnarefren} the isomorphism therefore lifts to an isomorphism between $\FK'(A)$ and $\FK'(B)$, and by~\cite[5.14]{filtrated} this lifts to a $\KK(X_1)$-equivalence.
	\end{proof}
\end{corollary}

\begin{definition}
Let $A$ and $B$ be unital \csa s over $X$.  Then $ \phi \ : \ \FK ( A ) \rightarrow \FK ( B ) $ is a \emph{homomorphism that preserves the unit} if $\phi$ is a homomorphism of $\NT$-modules and $\phi_{X} ( [ 1_{A} ] ) = [ 1_{B} ]$ in $\FK_{X} ( A ) = \FK_{X} ( B )$.  We say that $\phi$ is an \emph{isomorphism that preserves the unit} if $\phi$ is an isomorphism of $\NT$-modules that preserves the unit.
\end{definition}

Combining this with~\cite[4.3]{kirchberg} and~\cite[2.1,3.2]{restorffruiz}, we obtain the following corollary.
\begin{corollary} \label{classi1}
	Let $A$ and $B$ be separable, nuclear, purely infinite \csa s that are tight over $X_1$ and whose simple subquotients lie in the bootstrap class.  Assume that $A$ has real rank zero.
\begin{itemize}	
\item[(1)]  If $A$ and $B$ are stable, then every isomorphism from $\FK(A)$ to $\FK(B)$ can be lifted to a \siso{} from $A$ to $B$. 

\item[(2)] If $A$ and $B$ are unital, then every isomorphism from $\FK(A)$ to $\FK(B)$ that preserves the unit can be lifted to a \siso{} from $A$ to $B$. 
\end{itemize}
\end{corollary}

\begin{remark} \label{Xoneop}
The space $X_4=X_1^{\op}$ has been studied in~\cite{rasmusmanuel}
where it is shown that the indecomposable transformations for $X_1^\op$ are
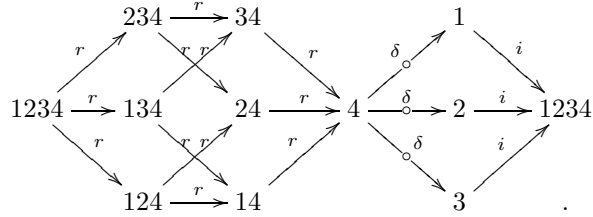
\[ \xymatrix@!0@C=40pt@R=35pt{
& 234\ar[r]^-{r}\ar[dr]^-{r} & 34\ar[dr]^-{r} & & 1\ar[dr]^-{i} & \\
1234\ar[ur]^-{r}\ar[r]^-{r}\ar[dr]^-{r} & 134\ar[ur]^-{r}\ar[dr]^-{r} & 24\ar[r]^-{r} & 4\ar[ur]^-{\delta}|-\circ\ar[r]^-{\delta}|-\circ\ar[dr]^-{\delta}|-\circ & 2\ar[r]^-{i} & 1234 \\
& 124\ar[r]^-{r}\ar[ur]^-{r} & 14\ar[ur]^-{r} & & 3\ar[ur]^-{i} & .
}  \]
It is straightforward to copy the methods of Meyer and Nest in~\cite{filtrated} to construct a refined filtrated $\K$-theory for which there is a UCT; for $X_1^{\op}$ the extra representing object is the mapping cone of a generator of $\NT(14,234)$.
The methods we used for the spaces $X_1$ apply to $X_1^{\op}$ as well since the boundary maps $\delta$ are placed in similar places in the structure diagrams for $\NT$ of $X_1^{\op}$.
\end{remark}

\section{Another counterexample}

Consider the space $X_2=\{1,2,3,4\}$ with $\Op(X_2)=\{\emptyset,4,34,234,134,X_2\}$.  Then $1<3$, $2<3$ and $3<4$, $\LC(X_2)^*=\{4,34,234,134,1234,3,23,13,123,1,2\}$, and its Hasse diagram is
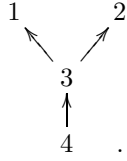
\[ \xymatrix@!0@C=20pt@R=25pt{
1 && 2 \\
& 3\ar[ul]\ar[ur] & \\
& 4\ar[u] & .
} \]
The indecomposable transformations in the category $\NT$ have been studied in detail in~\cite[6.1.2]{rasmus} and are the maps in the following diagram:
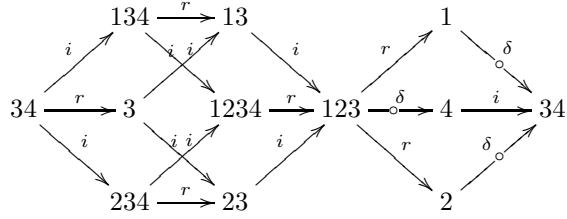
\[ \xymatrix@!0@C=40pt@R=35pt{
 & 134\ar[r]^-{r}\ar[dr]^-{i} & 13\ar[dr]^-{i} &  & 1\ar[dr]|\circ^-{\delta} &  \\
34\ar[r]^-{r}\ar[ur]^-{i}\ar[dr]^-{i} & 3\ar[ur]^-{i}\ar[dr]^-{i} & 1234\ar[r]^-{r} & 123\ar[r]|\circ^-{\delta}\ar[ur]^-{r}\ar[dr]^-{r} & 4\ar[r]^-{i} & 34 \\
 & 234\ar[r]^-{r}\ar[ur]^-{i} & 23\ar[ur]^-{i} &  & 2\ar[ur]|\circ^-{\delta} & 
} \]
As with the first counterexample, there exists a refinement $\FK'$ of $\FK$ for which there is a UCT, cf.~\cite[6.1]{rasmus}, so for $A$ and $B$ in the bootstrap class $\B(X_2)$ one can lift an isomorphism between $\FK'(A)$ and $\FK'(B)$ to a $\KK(X_2)$-equivalence.

For $X_2$ one constructs an extra representing object $R_{12334}$ as the mapping cone of a generator of the cyclic free group $\NT(23,134)$, and its filtrated $\K$-theory is then
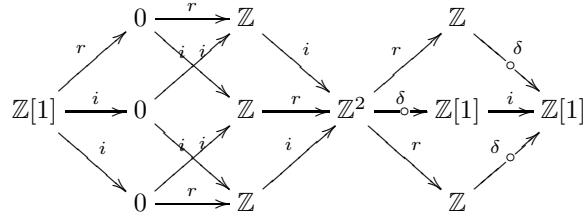
\[ \xymatrix@!0@C=40pt@R=35pt{
& 0\ar[r]^-{r}\ar[rd]^-{i} & \Z\ar[rd]^-{i} & & \Z\ar[dr]^-{\delta}|\circ & \\
\Z[1]\ar[ru]^-{r}\ar[r]^-{i}\ar[rd]^-{i} & 0\ar[ru]^-{i}\ar[rd]^-{i} & \Z\ar[r]^-{r} & \Z^2\ar[ru]^-{r}\ar[r]^-{\delta}|\circ\ar[rd]^-{r} & \Z[1]\ar[r]^-{i} & \Z[1] \\
& 0\ar[ru]^-{i}\ar[r]^-{r} & \Z\ar[ru]^-{i} & & \Z\ar[ur]^-{\delta}|\circ &
 } \]
where the three maps $i_{13}^{123}$, $r_{1234}^{123}$ and $i_{23}^{123}$ are given by the three coordinate embeddings $\Z\to\Z^3/(1,1,1)$, the three maps $r_{123}^1$, $\delta_{123}^4$ and $r_{123}^2$ are given by the three projections $\Z^3/(1,1,1)\to\Z^2/(1,1)$ onto coordinate hyperplanes, and the three maps $\delta_1^{34}$, $i_4^{34}$ and $\delta_2^{34}$ are the identity.

Since $\pd(\FK(R_{12334}))=1$, we see that for any \csa\ $A$ over $X_2$ that lies in the bootstrap class over $X_2$, $A$ and $R_{12334}$ will be $\KK(X_2)$-equivalent if and only if the groups $\FK_Y(A)$ and $\FK_Y(R_{12334})$ are isomorphic for all $Y\in\LC(X_2)^*$ and the natural transformation $\FK_{13}(A)\oplus\FK_{1234}(A)\to\FK_{123}(A)$ is an isomorphism, cf.~Lemma~\ref{thankyourasmus}.
The indecomposable transformations in the ring $\NT'$ fit into the following diagram:
\begin{equation}
 \xymatrix@!0@C=40pt@R=35pt{
 & 134\ar[dr] & & 13\ar[dr]^-{i} &  & 1\ar[dr]|\circ^-{\delta} &  \\
34\ar[r]^-{r}\ar[ur]^-{i}\ar[dr]^-{i} & 3\ar[r] & 12334\ar[ru]\ar[r]\ar[rd] & 1234\ar[r]^-{r} & 123\ar[r]|\circ^-{\delta}\ar[ur]^-{r}\ar[dr]^-{r} & 4\ar[r]^-{i} & 34 \\
 & 234\ar[ur] & & 23\ar[ur]^-{i} &  & 2\ar[ur]|\circ^-{\delta} & 
} \label{diagram2}
\end{equation}

\subsection{The refined invariant}

\begin{lemma} \label{kegletrick12}
	There exists an exact triangle
\[ \xymatrix@!0@C=40pt@R=35pt{
R_{123}\ar[dr]_-{\phi} & & R_{12334}\ar[ll]_-{\pi}|\circ \\
& R_{13}\oplus R_{1234}\oplus R_{23}\ar[ur]_-{\iota} & 
} \]
	satisfying that $\bar\phi = (i_{13}^{123},r_{1234}^{123},i_{23}^{123})$,
	that $\bar\pi$ generates the group $\NT(123,12334)$, and that $\bar\iota=(f^{13},f^{1234},f^{23})$ with $f^Y$ generating $\NT(12344,Y)$.
	\begin{proof}
	Let $\phi\colon R_{123}\to R_{13}\oplus R_{1234}\oplus R_{23}$ be given by inclusion respectively restrictions to subspaces, such that 
	$\bar\phi = (i_{13}^{123},r_{1234}^{123},i_{23}^{123})$. 
	The proof is similar to the proof of Lemma~\ref{kegletrick1}.  Here $\FK(R_{123})$ is used to establish that $\bar\pi$ is a generator, and $\FK_Y$ is used for $f^Y$.
	\end{proof}
\end{lemma}

\begin{lemma} \label{kegletrick22}
	There exists an exact triangle
\[ \xymatrix@!0@C=40pt@R=35pt{
R_{12334}\ar[dr]_-{\pi} & & R_{34}\ar[ll]_-{\iota}|\circ \\
& R_{134}\oplus R_{3}\oplus R_{234}\ar[ur]_-{\phi} & 
} 
\]
satisfying that $\bar\phi = (i_{34}^{134},r_{34}^{3},i_{34}^{234})$,
	that $\bar\iota$ generates $\NT(12334,34)$, and that $\bar\pi=(f_{134},f_{3},f_{234})$ with each $f_Y$ generating the group $\NT(Y,12344)$ respectively.
\begin{proof}
	Let $\phi\colon R_{134}\oplus R_3\oplus R_{234}\to M_3(R_{34})$ be given by inclusions respectively restriction to a subspace, such that 
	$\bar\phi = (i_{34}^{134},r_{34}^{3},i_{34}^{234})$.
	The proof is similar to the proof of Lemma~\ref{kegletrick2}.  Here $\FK_4$ is used to establish that $\bar\iota$ is a generator, and $\FK_Y$ is used for $f_Y$.
\end{proof}
\end{lemma}

\begin{lemma} \label{gunnarsseqs2}
	There exist natural transformations $f_{134},f_{3},f_{234},f^{13},f^{1234},f^{23}$ 
	such that $\left<f_Y\right>=\NT(Y,12334)$ and $\left<f^Y\right>=\NT(12334,Y)$ and such that the sequences
\[ \xymatrix@!0@C=60pt@R=35pt{
\FK_{123}(A)\ar[rr]^-{f_{134} i_{4}^{134}\delta_{123}^4}|\circ & & \FK_{12334}(A)\ar[dl]^-{(f^{13}, f^{1234},f^{23})} \\
& \FK_{13}(A)\oplus\FK_{1234}(A)\oplus\FK_{23}(A)\ar[ul]^-{(i_{13}^{123},r_{1234}^{123},i_{23}^{123})} & 
} 
\] and \[
\xymatrix@!0@C=60pt@R=35pt{
\FK_{12334}(A)\ar[rr]^-{r_4^{34}\delta_{123}^4i_{23}^{123}f^{23}}|\circ & & \FK_{34}(A)\ar[dl]^-{(i_{34}^{134},r_{34}^3,i_{34}^{234})} \\
& \FK_{134}(A)\oplus\FK_{3}(A)\oplus\FK_{234}(A)\ar[ul]^-{(f_{134},f_{3},f_{234})} & 
} 
\]	
	are exact for all \csa s $A$ over $X_2$.
\begin{proof}
	This follows from Lemmas~\ref{kegletrick12} and~\ref{kegletrick22} by applying $\KK_*(X_2;-,A)$ and using that by the Diagram~\eqref{diagram2} the transformation
	$f_{134}i_4^{134}\delta_{123}^4$
	generates $\NT(123,12334)$
	and the transformation
	$r_4^{34}\delta_{123}^4i_{23}^{123}f^{23}$
	generates $\NT(12334,34)$.
\end{proof}
\end{lemma}

\subsection{A classification result}

A slightly more general result, like the result in Section~\ref{class_first}, can be obtained, but we state a weaker result to ease notation.

\begin{proposition} \label{gunnarefren2}
	Let $A$ and $B$ be \csa s over $X_2$ and assume that $A$ and $B$ have real rank zero.
	Then any homomorphism $\phi\colon\FK(A)\to\FK(B)$ can be uniquely extended to a homomorphism $\phi'\colon\FK'(A)\to\FK'(B)$.
	Furthermore, if $\phi$ is an isomorphism then so is $\phi'$.
\begin{proof}
	The proof is similar to the proof of Theorem~\ref{gunnarefren}.  Since $A$ and $B$ have real rank zero, $\delta_{123}^4\colon\FK_{123}^0(A)\to\FK_4^1(A)$  and $\delta_{123}^4\colon\FK_{123}^0(B)\to\FK_4^1(B)$ vanish, so by Lemma~\ref{gunnarsseqs2} the horizontal sequences in the following diagram are exact 

\[ \xymatrix{
	0\ar[r] & \FK_{12334}^{1}(A)\ar[r]\ar@{..>}[d]^-{\phi_{12334}^{1}} & \FK_{13}^{1}(A)\oplus\FK_{1234}^{1}(A)\oplus\FK_{23}^{1}(A)\ar[r]\ar[d]^-{\phi_{13}^{1}\oplus\phi_{1234}^{1}\oplus\phi_{23}^{1}} & \FK_{123}^{1}(A)\ar[d]^-{\phi_{123}^{1}} \\
	0\ar[r] & \FK_{12334}^{1}(B)\ar[r] & \FK_{13}^{1}(B)\oplus\FK_{1234}^{1}(B)\oplus\FK_{23}^{1}(B)\ar[r] & \FK_{123}^{1}(B)
} \]
\[ \xymatrix{
	\FK_4^0(A)\ar[r]\ar[d]^-{\phi_4^0} & \FK_{134}^0(A)\oplus\FK_3^0(A)\oplus\FK_{234}^0(A)\ar[r]\ar[d]^-{\phi_{134}^0\oplus\phi_3^0\oplus\phi_{234}^0} & \FK_{12334}^0(A)\ar[r]\ar@{..>}[d]^-{\phi_{12334}^0} & 0 \\ 
		\FK_4^0(B)\ar[r] & \FK_{134}^0(B)\oplus\FK_3^0(B)\oplus\FK_{234}^0(B)\ar[r] & \FK_{12334}^0(B)\ar[r] & 0 
} \]
so we may recover $\FK_{12334}^1$ as the kernel of $(i_{13}^{123},r_{1234}^{123},i_{13}^{123})$ and $\FK_{12334}^0$ as the cokernel of $(i_{34}^{134},r_{34}^3,i_{34}^{234})$, as in the proof of Theorem~\ref{gunnarefren}.
\end{proof}
\end{proposition}

\begin{corollary} \label{isoslift2}
	Let $A$ and $B$ be \csa s in the bootstrap class over $X_2$ and assume that $A$ has real rank zero.  Then any isomorphism between $\FK(A)$ and $\FK(B)$ lifts to a $\KK(X_2)$-equivalence.
\begin{proof}
	Since $\FK(A)$ and $\FK(B)$ are isomorphic, $\delta_{123}^4\colon\FK_{123}^0(B)\to\FK_4^1(B)$ vanishes, so the proof of Proposition~\ref{gunnarefren2} applies, hence the isomorphism lifts to an isomorphism between $\FK'(A)$ and $\FK'(B)$ and by~\cite[6.1.22]{rasmus} this lifts to a $\KK(X_2)$-equivalence.
\end{proof}
\end{corollary}

\begin{corollary} \label{classi2}
	Let $A$ and $B$ be separable, nuclear, purely infinite \csa s that are tight over $X_2$ and whose simple subquotients lie in the bootstrap class.  Assume that $A$ has real rank zero.
\begin{itemize}	
\item[(1)] If $A$ and $B$ are stable, then every isomorphism from $\FK(A)$ to $\FK(B)$ can be lifted to a \siso{} from $A$ to $B$. 

\item[(2)] If $A$ and $B$ are unital, then every isomorphism from $\FK(A)$ to $\FK(B)$ that preserves the unit can be lifted to a \siso{} from $A$ to $B$.
\end{itemize} 
\end{corollary}

\begin{remark}\label{Xtwoop}
The space $X_5=X_2^{\op}$ has been studied in~\cite{rasmusmanuel}
where it is shown that the indecomposable transformations for $X_2^\op$ are
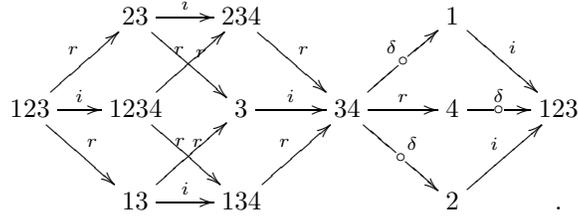
\[ \xymatrix@!0@C=40pt@R=35pt{
& 23\ar[r]^-{i}\ar[dr]^-{r} & 234\ar[dr]^-{r} && 1\ar[dr]^-{i} & \\
123\ar[ur]^-{r}\ar[r]^-{i}\ar[dr]^-{r} & 1234\ar[ur]^-{r}\ar[dr]^-{r} & 3\ar[r]^-{i} & 34\ar[ur]^-{\delta}|-\circ\ar[r]^-{r}\ar[dr]^-{\delta}|-\circ & 4\ar[r]^-{\delta}|-\circ & 123 \\
& 13\ar[ur]^-{r}\ar[r]^-{i} & 134\ar[ur]^-{r} & & 2\ar[ur]^-{i} & .
} \]
As with $X_1^{\op}$, cf.~Remark~\ref{Xoneop}, it is straightforward to copy the methods of Meyer and Nest in~\cite{filtrated} to construct a refined filtrated $\K$-theory for which there is a UCT; for $X_2^{\op}$ the extra representing object is the mapping cone of a generator of $\NT(134,23)$.
And as with $X_1^{\op}$, the methods we used for the spaces $X_1$ and $X_2$ apply to $X_2^{\op}$ since the boundary maps $\delta$ are placed in similar places in the structure diagrams for $\NT$ of $X_2^{\op}$.
\end{remark}

\section{A third counterexample} \label{secthm2}

Consider the space $X_3=\{1,2,3,4\}$ with $\Op(X_3)=\{\emptyset,4,24,34,234,X_3\}$.  Then $1<2$, $1<3$, $2<4$, $3<4$, $\LC(X_3)^*=\{4,24,34,234,1234,123,12,13,1,2,3\}$ and its Hasse diagram is
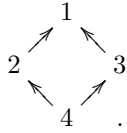
\[ \xymatrix@!0@C=20pt@R=20pt{
& 1 & \\
2\ar[ur] && 3\ar[ul] \\
& 4\ar[ul]\ar[ur] & .
} \]
The indecomposable transformations in the category $\NT$ have been studied in detail in~\cite[6.2.2]{rasmus} and are displayed in the following diagram:
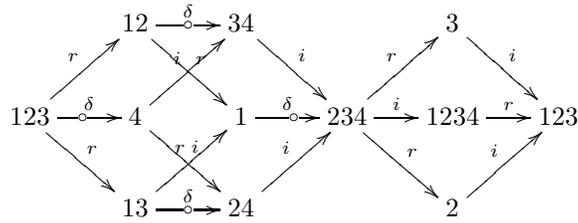
\[ \xymatrix@!0@C=40pt@R=35pt{
 & 12\ar[r]|\circ^-{\delta}\ar[dr]^-{r} & 34\ar[dr]^-{i} &  & 3\ar[dr]^-{i} &  \\
123\ar[r]|\circ^-{\delta}\ar[ur]^-{r}\ar[dr]^-{r} & 4\ar[ur]^-{i}\ar[dr]^-{i} & 1\ar[r]|\circ^-{\delta} & 234\ar[r]^-{i}\ar[ur]^-{r}\ar[dr]^-{r} & 1234\ar[r]^-{r} & 123 \\
 & 13\ar[r]|\circ^-{\delta}\ar[ur]^-{r} & 24\ar[ur]^-{i} &  & 2\ar[ur]^-{i} & 
} \]

The methods used for the spaces $X_1$ and $X_2$ do not apply to $X_3$ since the boundary maps $\delta$ are placed radically differently in the structure diagram for $\NT$ of $X_3$.
In fact, for this space $X_3$ there does exist tight, nuclear, separable, purely infinite \csa s $A$ and $B$ over $X_3$ of real rank zero that are not $\KK(X_3)$-equivalent but have isomorphic filtrated $\K$-theory.  

\begin{proof}[Proof of Theorem~\ref{thm2}]
The construction is similar to the one of R.~Meyer and R.~Nest in~\cite[p.~27ff]{filtrated} and some of the details are carried out in~\cite[6.2.4]{rasmus}.
Put $P_Y(Z)=\NT(Y,Z)$.  Consider the injective map $j\colon P_{234}\to P_{24}\oplus P_1[1]\oplus P_{34}$ induced by three generators of the groups $\NT(24,234)$, $\NT(1,234)$ and $\NT(34,234)$, and let $M$ denote the cokernel.  Let $k\geq 2$ and put $M_k=\tensor M{\Z/k}$.  Then $M_k$ is
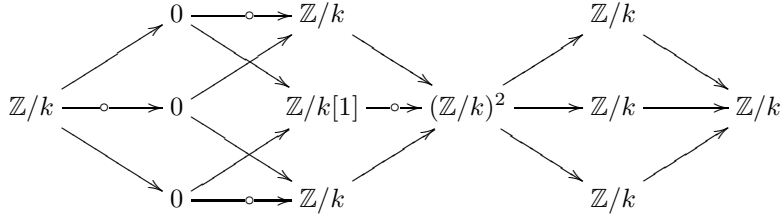
\[ \xymatrix@!0@C=55pt@R=35pt{
 & 0\ar[r]|\circ\ar[dr] & \Z/k\ar[dr] &  & \Z/k\ar[dr] &  \\
\Z/k\ar[r]|\circ\ar[ur]\ar[dr] & 0\ar[ur]\ar[dr] & \Z/k[1]\ar[r]|\circ & (\Z/k)^2\ar[r]\ar[ur]\ar[dr] & \Z/k\ar[r] & \Z/k \\
 & 0\ar[r]|\circ\ar[ur] & \Z/k\ar[ur] &  & \Z/k\ar[ur] & 
} \]
and has projective dimension 2, and
\[ \xymatrix@C-5pt{
0\ar[r] & P_{234}\ar[r] & P_{234}\oplus P_{24}\oplus P_1[1]\oplus P_{34} \ar[r] & P_{24}\oplus P_1[1]\oplus P_{34} \ar[r] & M_k \ar[r] & 0
} \]
is a projective resolution of $M_k$.
Notice that the boundary maps from even to odd parts in $M_k$ vanish.
There exists in the bootstrap class over $X_3$ a \csa\ $A_k$ with $\FK(A_k)=M_k$, see~\cite[6.2.4]{rasmus} for details.  Let
\[ \xymatrix@C-5pt{
Q_2\ar[r] & Q_1 \ar[r] & Q_0 \ar[r] & A_k
} \]
be a  $\ker\FK$-projective resolution which is a lift of the above projective resolution of $M_k$, and let
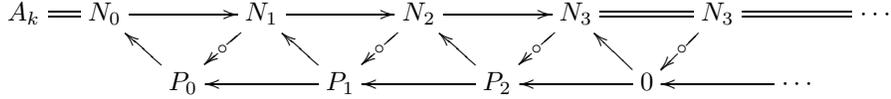
\[ \xymatrix@-12pt{
A_k\ar@{=}[r] & N_0 \ar[rr] && N_1\ar[dl]|\circ \ar[rr] && N_2\ar[dl]|\circ \ar[rr] && N_3\ar[dl]|\circ \ar@{=}[rr] && N_3\ar[dl]|\circ \ar@{=}[rr] && \cdots \\
&& P_0\ar[ul] && P_1\ar[ul]\ar[ll] && P_2\ar[ul]\ar[ll] && 0\ar[ul]\ar[ll] && \cdots\ar[ll] &&
} \]
be its phantom tower.  Then $N_2\cong_{\KK(X_3)} Q_2$ and the composite map $A_k\to N_2$ lies in $(\ker\FK)^2$.
Construct $B$ as the mapping cone of $A_k\to N_2$.
Then $B$ and $A_k\oplus SN_2$ are not $\KK(X_3)$-equivalent but have $\FK(B)\cong\FK(A_k)\oplus\FK(N_2)[1] = M_k\oplus P_{234}[1]$.
See~\cite[4.10, 5.5]{filtrated} for more details. 

Since all $\KK(X_3)$-equivalence classes in the bootstrap class over $X_3$ can be represented by tight, stable, nuclear, separable, purely infinite \csa s over $X_3$, cf.~\cite[4.6]{filtrated}, we can find such $C$ and $D$ with $C\cong_{\KK(X_3)}B$, $D\cong_{\KK(X_3)}A_k\oplus SN_2$ and $\FK(C)\cong\FK(D)\cong\FK(B)$.  Since $P_{234}[1]$ is
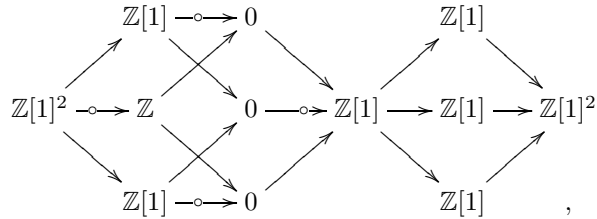
\[ \xymatrix@!0@C=40pt@R=35pt{
 & \Z[1]\ar[r]|\circ\ar[dr] & 0\ar[dr] &  & \Z[1]\ar[dr] &  \\
\Z[1]^2\ar[r]|\circ\ar[ur]\ar[dr] & \Z\ar[ur]\ar[dr] & 0\ar[r]|\circ & \Z[1]\ar[r]\ar[ur]\ar[dr] & \Z[1]\ar[r] & \Z[1]^2 \\
 & \Z[1]\ar[r]|\circ\ar[ur] & 0\ar[ur] &  & \Z[1]\ar[ur] & ,
} \]
we see that the boundary maps from even to odd parts in $\FK(B)$ vanish, so $C$ and $D$ will be of real rank zero as their simple subquotients are Kirchberg algebras and therefore of real rank zero, cf.~Observation~\ref{realrankzero}.
\end{proof}

\begin{remark}
The real rank zero counter-examples for the space $X_3$ have torsion in both even and odd degrees.
In~\cite{abk}, it is shown that for real rank zero \csa s over $X_3$ with free $\K_1$-groups, isomorphisms on a reduced filtrated $\K$-theory lift to $\KK(X_3)$-equivalences.
This reduced filtrated $\K$-theory is defined in~\cite{abk} by disregarding parts of the information in filtrated $\K$-theory, and it is equivalent to the reduced filtered $\K$-theory defined by the second named author in~\cite[4.1]{gunnar}. It is unknown whether isomorphisms on $\FK$ lift to $\KK(X_3)$-equivalences under these conditions. 
\end{remark}

\begin{remark}
The space $X_6$ has been studied in~\cite{rasmus} where R.~Bentmann fails to construct a finite refinement of filtrated $\K$-theory over $X_6$ that admits a UCT and remarks that it seems unlikely that such a finite refinement exists.  So our method cannot be applied for the space $X_6$.  In~\cite{rasmus}, R.~Bentmann constructs separable, stable, nuclear, purely infinite, tight \csa s $A$ and $B$ over $X_6$ that have isomorphic filtrated $\K$-theory and are not $\KK(X_6)$-equivalent.  One can check that the boundary map $\FK_1(A)\to\FK_3(A)$ does not vanish in either degrees, so neither $A$ and $B$ nor the suspensions $\Sus A$ and $\Sus B$ have real rank zero.  So there is so far no known real rank zero counter-example for $X_6$.
\end{remark}

\section{Acknowledgement}
This research was supported by the NordForsk Research Network ``Operator Algebras and Dynamics'' (grant \#11580), by the Faroese Research Council, and by the Danish National Research Foundation (DNRF) through the Centre for Symmetry and Deformation.

The second and third named authors are grateful to Professor S{\o}ren Eilers and the department of mathematics at the University of Copenhagen for providing the dynamic research environment where this work was initiated during the spring of 2010.
The first named author would like to thank Professor Ralf Meyer and the department of mathematics at Georg-August-Universit\"at G\"ottingen for kind hospitality during the spring and early summer of 2010.

%


\bibliography{referencer}

\end{document}